\documentclass{amsart}
\usepackage[utf8]{inputenc}
\usepackage{url}
\usepackage[usenames,dvipsnames]{color}

\theoremstyle{plain}
\newtheorem{thm}{Theorem}[section]

\newtheorem{rmk}[thm]{Remark}
\newtheorem{lem}[thm]{Lemma}

\newtheorem{cor}[thm]{Corollary}
\newtheorem{defn}[thm]{Definition}
\newtheorem{conj}[thm]{Conjecture}

\newcommand{\FF}{\mathbb{F}}
\newcommand{\QQ}{\mathbb{Q}}
\newcommand{\RR}{\mathbb{R}}
\newcommand{\ZZ}{\mathbb{Z}}
\newcommand{\frakg}{\mathfrak{g}}
\newcommand{\frakh}{\mathfrak{h}}
\renewcommand{\colon}{\!:}

\DeclareMathOperator{\an}{an}
\DeclareMathOperator{\coker}{coker}
\DeclareMathOperator{\cont}{ct}
\DeclareMathOperator{\Cont}{Ct}

\DeclareMathOperator{\Hom}{Hom}
\DeclareMathOperator{\id}{id}
\DeclareMathOperator{\image}{image}

\title{A note on the cohomology of $p$-adic analytic group actions}
\author{Annie Carter and Kiran S. Kedlaya}
\date{June 9, 2023}
\thanks{Thanks to Ruochuan Liu for suggesting Theorem~\ref{T:main}(b). Both authors were supported by NSF grant DMS-1802161. Kedlaya was additionally supported by NSF grant DMS-2053473 and by the UC San Diego Warschawski Professorship.}

\begin{document}

\begin{abstract}
We prove that given an analytic action of a compact $p$-adic Lie group on a Banach space over a field of positive characteristic, one can detect either the simultaneous vanishing or the simultaneous finite-dimensionality of all of the continuous cohomology groups from the corresponding statement for the restriction to a pro-$p$ procyclic subgroup.
We also formulate a conjecture generalizing this result, in which the base field is allowed to have mixed characteristic and the subgroup is allowed to be nilpotent.
Finally, we formulate an analogous conjecture about Lie algebra cohomology and relate this to a theorem of Kostant.
\end{abstract}

\maketitle

\section{Introduction}

Throughout this paper, fix a prime $p>0$. For $\Gamma$ a compact $p$-adic Lie group,
a detailed study of continuous $\Gamma$-actions on topological $\ZZ_p$-modules was made by Lazard \cite{Lazard}, including the crucial definition of an \emph{analytic} $\Gamma$-action. This concept (or more precisely an analogue thereof; see Definition~\ref{D:c-analytic})
was used in \cite[Theorem~4.1]{KedlayaHS} to prove a theorem asserting that
in certain cases, the vanishing of cohomology on a subgroup propagates to the whole group; this was called the \emph{Hochschild--Serre property} in \cite{KedlayaHS}, as in the case of a normal inclusion the statement in question is a formal consequence of the existence of the Lyndon--Hochschild--Serre spectral sequence.

The first goal of this paper is to prove an extension of \cite[Theorem~4.1]{KedlayaHS} suggested to us by Ruochuan Liu, in which vanishing of cohomology is replaced by finite dimensionality (Theorem~\ref{T:main}(b)).
The method borrows some ideas from Kiehl's use of the Cartan--Serre method for proving finiteness of higher direct images in rigid analytic geometry \cite{Kiehl}.
In the process, we correct some errors in the proof of \cite[Theorem~4.1]{KedlayaHS}.

The second goal is to introduce some conjectural extensions of these results,
in both positive characteristic (Conjecture~\ref{conj:char p analytic conjecture})
and mixed characteristic (Conjecture~\ref{conj:char 0 analytic conjecture}).
Whereas the aforementioned results only apply when the subgroup is procyclic, we conjecture that similar results holds whenever the subgroup is nilpotent.
The mixed characteristic case is expected to be equivalent to a corresponding conjecture on Lie algebra cohomology (Conjecture~\ref{conj:lie algebra cohom}), which is known in certain cases thanks to a theorem of Kostant (Theorem~\ref{thm:Kostant}).

\section{Some nonarchimedean functional analysis}

Throughout this section, fix a nonarchimedean field $K$ of arbitrary characteristics.

\begin{defn} \label{D:strict}
A morphism $f\colon V \to V'$ of Banach spaces over $K$ is \emph{strict} if the subspace and quotient topologies on the image of $f$ coincide. 
Equivalently, the induced map $V \to \image(f)$ admits a bounded set-theoretic section.

A complex $V^\bullet$ of Banach spaces over $K$ is \emph{strict} if each differential $d^i\colon V^i \to V^{i+1}$
is strict. Equivalently, if we equip each cohomology group $H^i(V^\bullet)$ with the quotient topology induced by the subspace topology on $\ker(d^i)$ (or for short the \emph{subquotient topology}), then $V^\bullet$ is strict if the $H^i(V^\bullet)$ are all Hausdorff (or equivalently, all complete).
\end{defn}

\begin{thm}[Banach open mapping theorem] \label{T:open mapping theorem}
Any surjective morphism of Banach spaces over $K$ is strict.
\end{thm}
\begin{proof}
See \cite[Proposition~8.6]{SchneiderNFA}.
\end{proof}
\begin{cor} \label{C:open mapping finite cokernel}
Any morphism of Banach spaces over $K$ with finite-dimensional cokernel is strict.
\end{cor}
\begin{proof}
Let $f\colon V \to V'$ be a morphism of Banach spaces such that $\coker(f)$ is finite-dimensional. By choosing a basis for $\coker(f)$ and then lifting representatives to $V'$, we may construct a $K$-linear morphism $g: \coker(f) \to V'$; using the natural topology on $\coker(f)$ as a $K$-vector space (rather than the quotient topology), $g$ becomes a morphism of Banach spaces over $K$.
The morphism $f \oplus g: V \oplus \coker(f) \to V'$ is then a surjective morphism of Banach spaces over $k$, and hence by
Theorem~\ref{T:open mapping theorem} is strict. We may now obtain a bounded set-theoretic section of $f$ by taking a bounded set-theoretic section of $f \oplus g$ and then noticing that its restriction to $\image(f)$ maps into $V$.
\end{proof}
\begin{cor} \label{C:finite-dimensional cohomology means strict}
Let $V^\bullet$ be a complex of Banach spaces over $K$. If $H^i(V^\bullet)$ is finite-dimensional for all $i$, then $V^\bullet$ is strict.
\end{cor}
\begin{proof}
Apply Corollary~\ref{C:open mapping finite cokernel} to the map $V^i \to \ker(d^{i+1})$.
\end{proof}

\begin{defn}
A Banach space $V$ over $K$ is \emph{separable} if it contains a dense subspace of at most countable dimension.
\end{defn}

\begin{rmk} \label{R:separable union}
A commonly used observation is that if $V$ is a Banach space over $K$ and $W$ is a $K$-vector subspace of $V$ which is dense, then $V$ is the union
(not just the completed union) of the completions of all subspaces of $W$ of at most countable dimension. A similar assertion in algebraic geometry is that every ring is the union of $\ZZ$-subalgebras of finite type.
\end{rmk}

\begin{lem} \label{L:basis}
Let $V$ be a separable Banach space over $K$ of dimension $n$.
\begin{enumerate}
    \item[(a)] If $n < \infty$, then $V \cong K^n$.
    \item[(b)] If $n = \infty$, then $V$ is isomorphic to the space of null sequences in $K$
    equipped with the supremum norm.
\end{enumerate}
\end{lem}
\begin{proof}
See \cite[Proposition~10.4]{SchneiderNFA} or \cite[Lemma~1.3.8]{KedlayaBook}.
\end{proof}
\begin{cor} \label{C:kernel of base extension}
Let $K \to K'$ be an inclusion of nonarchimedean fields.
Let $V \to V'$ be a morphism of Banach spaces over $K$.
\begin{enumerate}
    \item[(a)]
The natural map
\[
\ker(V \to V') \mathbin{\widehat{\otimes}_K} K' \to \ker(V \mathbin{\widehat{\otimes}_K} K' \to V' \mathbin{\widehat{\otimes}_K} K')
\]
is an isomorphism of Banach spaces over $K'$. 
\item[(b)]
If $V \to V'$ is strict, then the natural map
\[
\coker(V \to V') \mathbin{\widehat{\otimes}}_K K' \to \coker(V \mathbin{\widehat{\otimes}}_K K' \to V' \mathbin{\widehat{\otimes}}_K K')
\]
is an isomorphism of Banach spaces over $K'$.
\item[(c)]
The map $V \to V'$ is strict if and only if $V \mathbin{\widehat{\otimes}}_K K' \to V' \mathbin{\widehat{\otimes}}_K K'$ is strict.
\end{enumerate}
\end{cor}
\begin{proof}
By Remark~\ref{R:separable union} we may reduce to the case where $K'$ is the completion of a countably generated field extension of $K$.
The claim then follows by applying Lemma~\ref{L:basis} to $K'$ viewed as a Banach space over $K$.
(Compare \cite[Lemma~2.2.9]{KedlayaLiu}.)
\end{proof}

\begin{lem} \label{L:cohomology after base change}
Let $K \to K'$ be an inclusion of nonarchimedean fields.
Let $V^\bullet$ be a strict complex of Banach spaces over $K$. 
Then for each $i$, the induced map 
\[
H^i(V^\bullet) \mathbin{\widehat{\otimes}}_K K' \to H^i(V^\bullet \mathbin{\widehat{\otimes}}_K K')
\]
is an isomorphism of Banach spaces over $K'$.
\end{lem}
\begin{proof}
Let $d^i\colon V^i \to V^{i+1}$ be the differential
and write $d^i \mathbin{\widehat{\otimes}}_K K'\colon V^i \mathbin{\widehat{\otimes}}_K K' \to V^{i+1} \mathbin{\widehat{\otimes}}_K K'$ for the induced differential. By Corollary~\ref{C:kernel of base extension}(a), the natural map
\[
\ker(d^i) \mathbin{\widehat{\otimes}}_K K' \to \ker(d^i \mathbin{\widehat{\otimes}}_K K')
\]
is an isomorphism of Banach spaces over $K'$. By the strictness condition and Corollary~\ref{C:kernel of base extension}(a,b),
the natural map
\[
\image(d^{i-1}) \mathbin{\widehat{\otimes}}_K K' \to \image(d^{i-1} \mathbin{\widehat{\otimes}}_K K')
\]
is an isomorphism of Banach spaces over $K'$.
We thus obtain isomorphisms
\begin{align*}
(V^{i-1} \mathbin{\widehat{\otimes}}_K K') / \ker(d^{i-1} \mathbin{\widehat{\otimes}}_K K') &\cong V^{i-1}/\ker(d^{i-1}) \mathbin{\widehat{\otimes}}_K K'\\
&\cong \image(d^{i-1}) \mathbin{\widehat{\otimes}}_K K' \\
&     \cong \image(d^{i-1} \mathbin{\widehat{\otimes}}_K K').
\end{align*}
This yields the desired result.
\end{proof}

\begin{lem} \label{L:split inclusion}
Let $V$ be a separable Banach space over $K$.
Let $V_0$ be a closed subspace of $V$.
\begin{enumerate}
\item[(a)]
The Banach space $V_0$ over $K$ is also separable.
\item[(b)]
There exists a closed subspace $V_1$ of $V$ such that the induced map $V_0 \oplus V_1 \to V$
is an isomorphism.
\end{enumerate}
\end{lem}
\begin{proof}
Since $V$ is separable, so is $V/V_0$.
In particular, we may apply Lemma~\ref{L:basis} to fix an isomorphism of $V/V_0$ with the Banach space of finite or countable null sequences over $K$ equipped with the supremum norm. We treat 
only the countable case, as the finite case is similar but easier.

Let $K^\infty$ denote the space of countable null sequences.
Let $e_1, e_2, \ldots \in V/V_0$ be the vectors corresponding to the standard basis vectors of $K^\infty$ via the chosen isomorphism $K^\infty \to V/V_0$. Fix $\epsilon> 0$ and lift each $e_i$ to an element $v_i \in V$ whose norm is at most $1+\epsilon$ times the quotient norm of $e_i$. Let $V_1$ be the completion of the span of the $v_i$; then the composition
\[
K^\infty \to V_1 \to V \to V/V_0,
\]
where the first map is $(a_i)_i \mapsto \sum_{i=1}^\infty a_i v_i$,
is our chosen isomorphism $(a_i)_i \mapsto \sum_{i=1}^\infty a_i e_i$. In particular, $K^\infty \to V_1 \to V$ is a strict inclusion, so $K^\infty \to V_1$ is an isomorphism, as then is $V_1 \to V/V_0$. This yields (b); since this in turn yields an isomorphism $V/V_1 \cong V_0$, we also deduce (a).
\end{proof}

\begin{cor} \label{C:split surjection}
Let $f\colon V \to V'$ be a morphism of separable Banach spaces with zero (resp. finite-dimensional) cokernel.
Then there exists a morphism $g\colon V' \to V$ such that $f \circ g - \id_{V'}$ is zero (resp. of finite rank).
If in addition $f$ has finite-dimensional kernel, then we may also ensure that $g \circ f - \id_V$ is of finite rank.
\end{cor}
\begin{proof}
If $f$ has zero cokernel, then the claim follows by applying Lemma~\ref{L:split inclusion} to the subspace $\ker(f)$ of $V$.
Otherwise, the map $V' \to \coker(f)$ is split;
choosing a splitting, we obtain a surjection $f'\colon V \oplus \coker(f) \to V'$ to which we may apply the previous case
to obtain a map $g'\colon V' \to V \oplus \coker(f)$ such that $f' \circ g' = \id_{V'}$. Writing $\pi\colon V \oplus \coker f \to V$ for the canonical projection and taking $g = \pi \circ g'$, we then have $f \circ g - \id_{V'} = (f \circ \pi - f') \circ g'$, which factors through $\coker f$, so it has finite rank.

If $f$ also has finite-dimensional kernel, then for $g$ as above, 
\[
f \circ (g \circ f - \id_{V}) = (f \circ g - \id_{V'}) \circ f
\]
also has finite rank. We may then deduce the same for $g \circ f - \id_V$.
\end{proof}

\begin{cor} \label{C:split surjection1}
Let $f\colon V \to V'$ be a strict morphism of separable Banach spaces.
Then there exists a morphism $g\colon V'	\to V$ such that $f \circ g - \id_{V'}$ and $g \circ f - \id_V$ are both strict.
\end{cor}
\begin{proof}
By Lemma~\ref{L:split inclusion}, the inclusions
\[
\ker(f) \subseteq V, \qquad \image(f) \subseteq V'
\]
of Banach spaces (the second one following from strictness) are both split.
Let $g_1\colon V' \to \image(f)$ be a map splitting the second inclusion and $g_2\colon V \to \ker(f) \subseteq V$ be a map splitting the first inclusion.
The map $g_2$ is an idempotent with image $\ker(f)$; its kernel $W$ is a complement of $\ker(f)$ in $V$,
so $W$ is naturally isomorphic to $V/\ker(f) \cong \image(f)$. We now interpret $g_1$ as a map $V' \to W \subset V$
and then take $g = (\id_V - g_2) \circ g_1$.
The map $f \circ g - \id_{V'}$ is a strict surjection onto $\image(\id_{V'} - g_1)$.
The map $g \circ f - \id_V$ is a strict surjection onto $\ker(f)$.
\end{proof}

\section{Analytic group actions}

Throughout this section, let $\Gamma$ be a compact $p$-adic Lie group.

\begin{defn} \label{D:coordinate chart}
By a \emph{coordinate chart} for $\Gamma$,
we will mean a homeomorphism $\pi\colon \ZZ_p^n \cong U$ of topological spaces for some neighborhood $U$ of the identity in $\Gamma$, where $n$ is the rank of $\Gamma$; we further require that $\pi(0)$ is the identity in $\Gamma$.
Given such a chart, let $\Gamma_j$ denote the subset of $\Gamma$ corresponding to $p^j \ZZ_p^n \subseteq \ZZ_p^n$.

In practice, we will often assume that $\pi$ has been chosen so that the subsets $\Gamma_j$ are all subgroups. This will be harmless on account of Lemma~\ref{L:commutators} and Remark~\ref{R:value of j} below.
\end{defn}

\begin{lem} \label{L:commutators}
Fix a coordinate chart $\pi\colon \ZZ_p^n \cong U$ for $\Gamma$.
Then for $j \gg 0$, the following statements hold.
\begin{enumerate}
    \item[(a)] The group $\Gamma_j$ is a subgroup of $\Gamma$.
    \item[(b)] The group $\Gamma_j/\Gamma_{j+1}$ is isomorphic to $(\ZZ/p\ZZ)^n$.
    \item[(c)] For $j_1, j_2 \geq j$,
    \[
    [\Gamma_{j_1}, \Gamma_{j_2}] \subseteq \Gamma_{j_1+j_2-j}.
    \]
\end{enumerate}
\end{lem}
\begin{proof}
Let $L$ be the Lie algebra (over $\QQ_p$) associated to $\Gamma$.
Then $L$ is also the Lie algebra associated to an algebraic group $\Gamma_{\QQ_p}$ over $\QQ_p$ in which $\Gamma$ occurs as a compact open subgroup (of the group of $\QQ_p$-points). The exponential map then defines
a homeomorphism between some lattice $T$ in $L$ and some compact neighborhood of the identity in $\Gamma$.

Via the exponential map, the group operation corresponds to some locally analytic map $T \times T \to L$.
By shrinking $T$, we can ensure that this is in fact an analytic map. That is, if we fix a basis $e_1,\dots,e_n$ of $T$, then the group operation is expressed by some functions
\[
f_1, \dots, f_n\colon \ZZ_p^{2n} \to \QQ_p
\]
which can be represented as convergent power series. Thanks to the behavior of the identity element (which corresponds to $0 \in T$), we see that
\begin{equation} \label{eq:formula for Lie addition}
f_i(x_1,\dots,x_n,y_1,\dots,y_n) \equiv x_i + y_i \pmod{(x_1,\dots,x_n,y_1,\dots,y_n)^2}.
\end{equation}
In particular, for any fixed $c>0$, for $j \gg 0$ 
we have
\[
f_i(x_1,\dots,x_n,y_1,\dots,y_n) \equiv x_i + y_i \pmod{p^{j+c}} \quad (x_1,\dots,x_n,y_1,\dots,y_n \in p^j \ZZ_p).
\]
From this we may deduce (a) and (b).

To deduce (c), we induct on $(j_1-j) + (j_2-j)$ (uniformly over choices of $j$); 
when this sum is zero we reduce immediately to (a).
To treat the induction step, suppose without loss of generality that $j_1 > j$. By (b), each element $\gamma_1 \in \Gamma_{j_1}$ is the $p$-th power of some element $\eta_1$ of $\Gamma_{j_1-1}$. For $\gamma_2 \in \Gamma_{j_2}$, by the induction hypothesis we have
\[
\eta_1^{-1} \gamma_2 \eta_1 = \gamma_2 \xi \mbox{ for some } \xi \in \Gamma_{j_1+j_2-j-1}.
\]
Consequently,
\begin{align*}
    \gamma_1^{-1} \gamma_2 \gamma_1 &= \eta_1^{-p} \gamma_2 \eta_1^p \\
    &= \gamma_2 \xi (\eta_1^{-1} \xi \eta_1) (\eta_1^{-2} \xi \eta_1^2) \cdots (\eta_1^{1-p} \xi \eta_1^{p-1}).
\end{align*}
By (b), the group $\Gamma_{j_1+j_2-j-1}/\Gamma_{j_1+j_2-j}$ is an elementary abelian $p$-group.
By the induction hypothesis (applied with $j = j_1$), 
the elements $\eta_1^{-i} \xi \eta_1^i$ for $i=0,\dots,p-1$ all represent the \emph{same} 
class in $\Gamma_{j_1+j_2-j-1}/\Gamma_{j_1+j_2-j}$, so their product represents the zero class.
\end{proof}

\begin{rmk} \label{R:value of j}
In Lemma~\ref{L:commutators}, replacing $U$ with the image of $p\ZZ_p^n$ gives a new chart for which the conclusion of (c) holds with $j$ replaced by $j-1$. Consequently, we can always choose a chart for which the conclusion of Lemma~\ref{L:commutators} holds with $j=0$.
\end{rmk}

\begin{defn}
Note that in the proof of Lemma~\ref{L:commutators}, the exponential map defines a map $\pi'\colon T \to \Gamma$ which itself becomes a coordinate chart upon fixing a group homeomorphism $\ZZ_p^n \cong T$. We say that the original coordinate chart $\pi$ is \emph{analytic} if for some (hence any sufficiently large) $m$,
the map $\pi^{-1} \circ \pi'\colon p^m \ZZ_p^n \to \ZZ_p^n$ is locally analytic.
\end{defn}

\begin{defn} \label{D:analytic}
Let $A$ be the completion of the group ring $\ZZ_p[\Gamma]$ with respect to the $p$-augmentation ideal
$\ker(\ZZ_p[\Gamma] \to \FF_p)$.
An \emph{analytic $\Gamma$-module} is a left $A$-module $M$ complete with respect to a valuation
$w(M; \bullet)$ for which there exist $a>0, c \in \RR$ such that
\begin{equation} \label{eq:analytic action}
w(M; xy) \geq aw(A; x) + w(M; y) + c \qquad (x \in A, y \in M).
\end{equation}
This implies that $M$ is a topological $A$-module for the topology defined by $w$.
If $M$ is $p$-torsion-free, it also implies that for any $x \in M$ the action map $\Gamma \to M$ taking $\gamma$ to $\gamma(x)$ is \emph{analytic} in the sense of being locally represented by a convergent power series on an analytic coordinate chart (see \cite[(V.2.3.6.2)]{Lazard}).
\end{defn}

In what follows, we will want to take $M$ not to be $p$-torsion-free, but rather a Banach module over a nonarchimedean field $K$ of characteristic $p$ (on which $\Gamma$ may or may not act).
In this case, Lazard's notion of analyticity must be replaced with the following.
\begin{defn} \label{D:c-analytic}
Let $M$ be a Banach module over a nonarchimedean field $K$ of characteristic $p$.
For $c > 1$, a function $f\colon \ZZ_p^n \to M$ is \emph{$c$-analytic} if there exists $d>0$ such that
\begin{equation} \label{eq:analyticity condition}
\left|f(\gamma_1,\dots,\gamma_n) - f(\gamma_1 + p^i\eta_1, \dots, \gamma_n + p^i\eta_n) \right|
\leq dc^{-p^i}
\end{equation}
for all $i \geq 0$ and all $\gamma_1,\dots,\gamma_n,\eta_1,\dots,\eta_n \in \ZZ_p$.
The $c$-analytic functions themselves form a Banach space for the norm given by the maximum of the supremum norm and the minimum value of $d$ for which \eqref{eq:analyticity condition} holds.
(This combination of metrics is loosely analogous to the definition of a Schwartz space in classical analysis. Compare the definition of a \emph{super-H\"older function} in \cite{BergerRosensztajn}.)

For example, if $M$ carries an analytic $\Gamma$-action and $\ZZ^n_p \cong U$ is an analytic coordinate chart for $\Gamma$, then for each $m \in M$ the action map $\ZZ^n_p \to U \to \Gamma \to M$
is $c$-analytic for some $c$ (depending on the value of $a$ in \eqref{eq:analytic action}, but not on $m$).
More precisely, there exist $c_0> 1$ and $e>0$ such that
\begin{equation} \label{eq:uniform analytic action}
\left\| \gamma - 1 \right\|_M \leq ec_0^{-p^i} \qquad (\gamma \in \Gamma_i).
\end{equation}
In case the action of $\Gamma$ is $K$-linear, we can restate this as follows:
a continuous $\Gamma$-action on $M$ is analytic if and only if for some $c>1$, the map
\[
\ZZ_p^n \to U \to \Gamma \to \Hom_K(M, M)
\]
is $c$-analytic, where $\Hom_K(M,M)$ denotes the space of bounded $K$-linear endomorphisms of $M$ equipped with the operator norm.
\end{defn}

\begin{defn}
For $M$ a linearly topologized complete $\Gamma$-module, let $C^\bullet_{\cont}(\Gamma, M)$ be the complex of continuous inhomogeneous cochains of $\Gamma$ valued in $M$. In particular, $C_{\cont}^i(\Gamma,M) = \Cont(\Gamma^i, M)$. The cohomology of this complex computes the continuous $\Gamma$-cohomology of $M$, or equivalently the continuous $A$-cohomology of $M$ \cite[(V.1.2.6)]{Lazard}.

Now fix an analytic coordinate chart $\pi:\ZZ_p^n \to U$ on $\Gamma$
and assume that $M$ is a Banach module over a nonarchimedean field $K$ of characteristic $p$
and the action of $\Gamma$ on $M$ is analytic.
For $c > 1$ and $i \geq 0$, an $i$-cochain $f\colon \Gamma^i \to M$ is \emph{$c$-analytic} if for 
any $(\gamma_1,\dots,\gamma_i) \in \Gamma^i$, the function
\[
(\ZZ_p^n)^i \to M, \qquad (\eta_1,\dots,\eta_i) \mapsto f(\gamma_1 \pi(\eta_1),\dots,\gamma_i \pi(\eta_i))
\]
is $c$-analytic.

Let $C^i_{\an,c}(\Gamma, M) \subseteq C^i_{\cont}(\Gamma, M)$ be the set of $c$-analytic
$i$-cochains. 
For $c_0$ as in Definition~\ref{D:c-analytic},
for $c \in (1, c_0]$ we obtain a subcomplex $C^\bullet_{\an,c}(\Gamma,M)$ of $C^\bullet_{\cont}(\Gamma,M)$.
\end{defn}

\begin{rmk}
As motivation for what follows, we explicitly compare the first cohomology groups of $C^{\bullet}_{\an,c}(\Gamma, M)$ and $C^{\bullet}_{\cont}(\Gamma,M)$. The latter computes the first continuous cohomology group of $\Gamma$ valued in $M$, which as usual is isomorphic to the group of continuous crossed homomorphisms modulo the principal crossed homomorphisms. Now note that every crossed homomorphism is automatically $c_0$-analytic (because it is uniquely determined by its values on a set of topological generators), so using $c$-analytic cochains for any $c \in (1,c_0]$ yields the same answer as using continuous cochains.
\end{rmk}

\begin{thm}[after Lazard] \label{T:Lazard char p}
Fix an analytic coordinate chart $\pi$ on $\Gamma$.
Let $M$ be a Banach module over a nonarchimedean field $K$ of characteristic $p$
equipped with an analytic $\Gamma$-action.
Define $c_0 > 0$ as in Definition~\ref{D:c-analytic}.
Then for any $c \in (1, c_0]$, there exists
a continuous map $s\colon C_{\cont}^\bullet(\Gamma,M) \to C_{\an,c}^{\bullet}(\Gamma,M)$
whose composition with the inclusion $\iota\colon C_{\an,c}^\bullet(\Gamma,M) 
\to C_{\cont}^\bullet(\Gamma,M)$ is homotopic to the identity
on $C_{\cont}^{\bullet}(\Gamma,M)$.
Consequently, $\iota$ induces surjective maps on cohomology groups.
\end{thm}
\begin{proof}
We follow the proof of \cite[Th\'eor\`eme V.2.3.10]{Lazard}, with some minor changes needed to accommodate the fact that we are not in a $p$-torsion-free setting.
Let $\gamma_1,\dots,\gamma_n \in U$ be the images under $\pi$ of the standard basis vectors of $\ZZ_p^n$. 
As in \cite[(V.2.3.5)]{Lazard},
let $X_\bullet$ be the \emph{completed standard complex} of $\Gamma$
and let $Y_\bullet$ be the \emph{quasi-minimal complex} of $\Gamma$;
these are both acyclic resolutions of $\ZZ_p$ in the category of $A$-modules (without topology).
By \cite[(V.1.1.5)]{Lazard}, there exist morphisms of filtered $A$-complexes
\[
\varphi_\bullet\colon X_\bullet \to Y_\bullet, \qquad \psi_\bullet\colon Y_\bullet \to X_\bullet
\]
whose compositions in either direction are homotopic to the identity.
In particular, there exist a family of morphisms $h_n\colon X_n \to X_{n+1}$ of filtered $A$-modules such that
\[
\psi_n \circ \varphi_n - 1 = d_{n+1} \circ h_n + h_{n-1} \circ d_n.
\]

Note that so far, none of this has anything to do with $M$.
We now recall that the continuous $\Gamma$-cohomology of $M$ can be computed as the
continuous $A$-cohomology of $M$ \cite[(V.1.2.6)]{Lazard},
which in turn coincides with the \emph{discrete} $A$-cohomology of $M$
\cite[(V.2.2.3)]{Lazard}. The latter is computed by the complex $\Hom_A(X_\bullet, M)$,
or thanks to the previous paragraph by the complex $\Hom_A(Y_\bullet, M)$. Now note that for any $f \in \Hom_A(Y_i, M)$, the $i$-cochain associated to $f \circ \varphi_n$ is $c_0$-analytic
(compare \cite[(V.2.3.6.4)]{Lazard}. By tracing through the identifications, we obtain from $\varphi_n$ the desired map $s$.
\end{proof}

\begin{rmk}
With a bit more work, it should be possible to adapt the rest of the proof of 
\cite[Th\'eor\`eme V.2.3.10]{Lazard} to show that the inclusion $\iota$ is in fact a quasi-isomorphism. As we will not need this here, we omit the details.

In the proof of Theorem~\ref{T:main}, we will use one special case of this stronger statement that can be verified directly: if $\Gamma$ is pro-$p$ and procyclic (that is, $\Gamma \cong \ZZ_p$) and $\gamma \in \Gamma$ is a topological generator, then $C^{\bullet}_{\cont}(\gamma, M)$  is quasi-isomorphic to the complex
$0 \to M \stackrel{\gamma-1}{\to} M \to 0$ with the nonzero terms placed in degrees $0$ and $1$.
(The cohomology of the latter defines a $\delta$-functor by the snake lemma, which is effaceable because
we can always embed $M$ into the module $\Cont(\Gamma, M)$ on which $\gamma-1$ is surjective; hence \cite[Tag 010T]{StacksProject} applies.)
\end{rmk}

\section{The main theorem}

In this section, we prove the following theorem. Part (a) reproduces \cite[Theorem~4.1]{KedlayaHS} but with some corrections made to the proof.
\begin{thm} \label{T:main}
Let $K$ be a nonarchimedean field of characteristic $p$.
Let $\Gamma$ be a compact $p$-adic Lie group.
Let $H$ be a closed, torsion-free, pro-$p$, procyclic subgroup of $\Gamma$.
Let $M$ be a Banach space over $K$ equipped with an action of $\Gamma$ which is continuous and analytic.
\begin{enumerate}
    \item[(a)] Suppose that $H^i_{\cont}(H, M) = 0$ for all $i \geq 0$.
    Then $H^i_{\cont}(\Gamma, M) = 0$ for all $i \geq 0$.
    \item[(b)] Suppose that $\Gamma$ acts $K$-linearly and that 
    $\dim_K H^i_{\cont}(H, M) < \infty$ for all $i \geq 0$.
    Then $\dim_K H^i_{\cont}(\Gamma, M) < \infty$ for all $i \geq 0$.
\end{enumerate}
\end{thm}

\begin{proof}[Proof of Theorem~\ref{T:main}]
Let $\eta$ be a topological generator of $H$. 
Since $K$ is of characteristic $p$, the actions of $\eta^p-1$ and $(\eta-1)^p$ on $M$ coincide;
consequently, the various hypotheses are preserved by replacing $H$ with an open subgroup.
In the other direction, 
we are free to prove the desired conclusion after replacing $\Gamma$ with an open normal subgroup $\Gamma_0$, as we may use the Hochschild--Serre spectral sequence to transfer the conclusion from $\Gamma_0$ back to $\Gamma$.
In particular, we may assume that:
\begin{itemize}
\item
$\Gamma$ itself is torsion-free and pro-$p$;
\item
 there exists an analytic coordinate chart of $\Gamma$
for which the conclusion of Lemma~\ref{L:commutators} holds with $j=0$
(see Remark~\ref{R:value of j}); and
\item
$\left\| \eta - 1 \right\|_M \in (0,1)$, where $\left\| \eta-1 \right\|_M$ denotes the operator norm of $\eta-1$ acting on $M$.
\end{itemize}

Since the map $\eta-1\colon M \to M$ has finite-dimensional cokernel, by Corollary~\ref{C:open mapping finite cokernel} it is strict.
In case (a), this means that $\eta-1$ has a bounded inverse $g$.
In case (b), 
Corollary~\ref{C:split surjection} implies the existence of a map $g\colon M \to M$ such that $(\eta-1) \circ g - \id_M$ and
$g \circ (\eta-1) - \id_M$ both have finite rank. 

Define $c_0 > 1$ as in Definition~\ref{D:c-analytic}, fix a nonnegative integer $l$ such that
\begin{equation} \label{eq:set up bound}
\max\{1, \left\| g \right\|_M\}^2 c_0^{-p^{l}} < 1,
\end{equation}
and put $c_l := c_0^{p^l}$.
Then for each cochain $f_n \in C^n_{\an,c_l}(\Gamma_l, M)$, there exists $m>0$ such that
\begin{equation} \label{eq:analytic cochain formula}
|f_n(\gamma_1,\dots,\gamma_n) - f_n(\gamma_1 \sigma_1,\dots,\gamma_n \sigma_n)| \leq m c_0^{-p^{\min\{i_1,\dots,i_n\}}}
\quad (\gamma_j \in \Gamma_l; i_j \geq l; \sigma_j \in \Gamma_{i_j}).
\end{equation}
As in Definition~\ref{D:c-analytic}, we obtain the correct topology on $C^n_{\an,c_l}(\Gamma_l, M)$ from the norm taking $f_n$ to the maximum of the supremum norm of $f_n$ and the minimum value of $m$ for which \eqref{eq:analytic cochain formula} holds.

By Theorem~\ref{T:Lazard char p}, we have
$H^i_{\cont}(\Gamma_l, M) \cong H^i_{\an,c_l}(\Gamma_l, M)$;
more precisely, the inclusion $\iota\colon C^\bullet_{\an,c_l}(\Gamma_l, M) \to C^\bullet_{\cont}(\Gamma_l,M)$ admits a section $s$
which composes with the inclusion to yield a map on $C^\bullet_{\cont}(\Gamma_l, M)$ homotopic to the identity.
In particular, $\iota$ induces surjective maps on cohomology groups; it thus suffices to check that
$C^\bullet_{\an,c_l}(\Gamma_l, M)$ has cohomology groups which are zero in case (a) or finite-dimensional in case (b),
as this will then imply the same for the groups $H^i_{\cont}(\Gamma_l, M)$ and (by Hochschild--Serre again) for the groups $H^i_{\cont}(\Gamma, M)$.

For $i \geq l$, define the chain homotopy $h_i$ from $C^\bullet_{\an,c_l}(\Gamma_l, M)$ to $C^\bullet_{\cont}(\Gamma_l, M)$ as follows:
for $f_n \in C^n_{\an,c_l}(\Gamma_l, M)$, set
\[
h_i(f_n)(\gamma_1,\dots,\gamma_{n-1}) = g^{p^i} \sum_{j=1}^n (-1)^{j-1} f_n(\gamma_1,\dots,\gamma_{j-1}, \eta^{p^i}, \gamma_j, \dots, \gamma_{n-1}). 
\]
We analyze this construction by computing that
\begin{align}
&(d \circ h_i + h_i \circ d - \iota)(f_n)(\gamma_1,\dots,\gamma_n) \nonumber \\
&\qquad = 
(\gamma_1 g^{p^i} - g^{p^i} \gamma_1) \sum_{j=1}^n (-1)^{j-1} f_n(\gamma_2, \dots, \gamma_j, \eta^{p^i}, \gamma_{j+1}, \dots, \gamma_n) \label{eq:homotopy commutator1}  \\
&\qquad \qquad - \sum_{j=1}^n g^{p^i} \left(f_n(\gamma_1,\dots,\gamma_{j-1},\eta^{p^i}\gamma_j, \gamma_{j+1},\dots,\gamma_n) \right. 
\nonumber \\
&\qquad \qquad \qquad \left. - f_n(\gamma_1,\dots,\gamma_{j-1},\gamma_j \eta^{p^i}, \gamma_{j+1},\dots,\gamma_n)\right) \label{eq:homotopy commutator2} \\
&\qquad \qquad + (g^{p^i} \eta^{p^i} - g^{p^i} - 1) f_n(\gamma_1, \dots, \gamma_n). \label{eq:homotopy commutator3}
\end{align}
To analyze the quantity \eqref{eq:homotopy commutator1}, note that the difference between
\[
\gamma_1 g^{p^i} - g^{p^i} \gamma_1 \qquad \mbox{and} \qquad g^{p^i} (\eta^{p^i} \gamma_1)(1 - \gamma_1^{-1} \eta^{-p^i} \gamma_1 \eta^{p^i}) g^{p^i}
\]
is zero in case (a) and of finite rank in case (b), and $\gamma_1^{-1} \eta^{-p^i} \gamma_1 \eta^{p^i} \in \Gamma_{i+k}$ if $\gamma_1 \in \Gamma_k$;
then use \eqref{eq:uniform analytic action}.
To analyze \eqref{eq:homotopy commutator2}, note that again $\gamma_j^{-1} \eta^{-p^i} \gamma_j \eta^{p^i} \in \Gamma_{i+k}$ if $\gamma_j \in \Gamma_k$,
then apply \eqref{eq:analytic cochain formula}. To analyze \eqref{eq:homotopy commutator3}, note that $g^{p^i} \eta^{p^i} - g^{p^i} - 1$ is zero in case (a) or of finite rank in case (b).
We deduce that $s \circ (d \circ h_i + h_i \circ d - \iota)$ equals zero (in case (a)) or a map of finite rank (in case (b)) plus
a map of operator norm at most
\begin{equation} \label{eq:bound on operator norm}
\max\{e,1\} \left\| s \right\|_M \max\{1, \left\| g \right\|_M\}^{2p^i} c^{-p^{i+l}}
\end{equation}
(where $e$ is as defined in \eqref{eq:uniform analytic action});
this tends to 0 as $i \to \infty$ by virtue of \eqref{eq:set up bound}.
We deduce that for $i \gg 0$, the map 
\[
s \circ (d \circ h_i + h_i \circ d - \iota) = d \circ s \circ h_i + s \circ h_i \circ d - 1\colon C^\bullet_{\an,c_l}(\Gamma_l, M) \to C^\bullet_{\an,c_l}(\Gamma_l, M)
\]
can be written as $\varphi+\epsilon$ where $\varphi$ has operator norm less than 1 and $\epsilon$ is zero (in case (a)) or a map of finite rank (in case (b)).

Now the map $1+\varphi\colon C^\bullet_{\an,c_l}(\Gamma_l, M) \to C^\bullet_{\an,c_l}(\Gamma_l, M)$ is an isomorphism of Banach spaces, so it induces an isomorphism of cohomology groups,
but is also homotopic to $\epsilon$. This 
implies the desired result in both cases.
\end{proof}

\section{Related questions}

We end by formulating a number of questions related to our main result. We start with a proposed generalization of Theorem~\ref{T:main} in which we weaken the procyclic hypothesis on $H$.

\begin{conj} \label{conj:char p analytic conjecture}
Let $K$ be a nonarchimedean field of characteristic $p$.
Let $\Gamma$ be a compact $p$-adic Lie group.
Let $H$ be a pro-nilpotent closed subgroup of $\Gamma$.
Let $M$ be a Banach space over $K$ equipped with an action of $\Gamma$ which is continuous and analytic.
\begin{enumerate}
    \item[(a)] Suppose that there exists a sequence of subgroups $H_j$ of $H$ forming a neighborhood basis of the identity such that $H^i_{\cont}(H_j, M) = 0$ for all $i \geq 0$ and all $j$.
    Then $H^i_{\cont}(\Gamma, M) = 0$ for all $i \geq 0$.
    \item[(b)] Suppose that $\Gamma$ acts $K$-linearly and that 
     there exists a sequence of subgroups $H_j$ of $H$ forming a neighborhood basis of the identity such that $\dim_K H^i_{\cont}(H_j, M) < \infty$ for all $i \geq 0$ and all $j$.
    Then $\dim_K H^i_{\cont}(\Gamma, M) < \infty$ for all $i \geq 0$.
\end{enumerate}
\end{conj}

\begin{rmk}
Note that in the case of Conjecture~\ref{conj:char p analytic conjecture} where $H$ is pro-$p$ procyclic, the hypothesis of either (a) or (b) is true as soon as it is true with $H_j = H$. This explains why Theorem~\ref{T:main} does not include a hypothesis quantified over a neighborhood basis of the identity in $H$.
\end{rmk}

We can also formulate an analogous conjecture in mixed characteristic.

\begin{conj} \label{conj:char 0 analytic conjecture}
Let $K$ be a nonarchimedean field of mixed characteristics $(0,p)$.
Let $\Gamma$ be a compact $p$-adic Lie group.
Let $H$ be a pro-nilpotent closed subgroup of $\Gamma$.
Let $M$ be a Banach space over $K$ equipped with a $K$-linear action of $\Gamma$ which is continuous and analytic.
\begin{enumerate}
    \item[(a)] Suppose that there exists a sequence of subgroups $H_j$ of $H$ forming a neighborhood basis of the identity such that $H^i_{\cont}(H_j, M) = 0$ for all $i \geq 0$ and all $j$.
    Then $H^i_{\cont}(\Gamma, M) = 0$ for all $i \geq 0$.
    \item[(b)] Suppose that there exists a sequence of subgroups $H_j$ of $H$ forming a neighborhood basis of the identity such that $\dim_K H^i_{\cont}(H_j, M) < \infty$ for all $i \geq 0$ and all $j$.
    Then $\dim_K H^i_{\cont}(\Gamma, M) < \infty$ for all $i \geq 0$.
\end{enumerate}
\end{conj}

\begin{rmk}
In contrast with Conjecture~\ref{conj:char p analytic conjecture},
in Conjecture~\ref{conj:char 0 analytic conjecture} the quantification over a neighborhood basis is required even in the case where $H$ is pro-$p$ procyclic. This comes down to the fact that in mixed characteristic we have $\gamma^p - 1 \neq (\gamma-1)^p$.
\end{rmk}

By adapting the proof of \cite[Th\'eor\`eme V.2.3.10]{Lazard} as in Theorem~\ref{T:Lazard char p}, it should be possible to reduce Conjecture~\ref{conj:char 0 analytic conjecture} to a corresponding statement at the level of Lie algebras. In order to formulate the conjecture, we first recall the definition of Lie algebra cohomology.

\begin{defn}
Let $\frakg$ be a Lie algebra over a field $K$ of characteristic $0$.
A (left) $\frakg$-module is defined as a left module for the universal enveloping algebra of $\frakg$ over $K$.
For $M$ a $\frakg$-module, the \emph{cohomology groups} $H^i(\frakg, M)$ are defined as the values on $M$
of the right derived functors of the functor of $\frakg$-invariants:
\[
M \mapsto \{m \in M\colon xm = 0 \quad \forall x \in \frakg\}.
\]
By analogy with group cohomology, these functors can be computed using a suitable complex of cochains called the
\emph{Chevalley--Eilenberg complex}. 

Let $\frakh \subseteq \frakg$ be a normal inclusion of Lie algebras
(meaning that the Lie bracket on $\frakg$ induces a well-defined map on the quotient $\frakg/\frakh$)
and let $M$ be a $\frakg$-module.
In analogy with the corresponding construction for groups, Hochschild and Serre \cite{HochschildSerreLie}
constructed a spectral sequence of the form
\[
E_2^{p,q} = H^p(\frakg/\frakh, H^q(\frakh, M)) \Longrightarrow H^{p+q}(\frakg, M).
\]
From the existence of this spectral sequence,
it follows that if $\frakh \subseteq \frakg$ is a subnormal inclusion and $M$ is a $\frakg$-module such that
$H^i(\frakh, M)$ is zero (resp.\ finite-dimensional) for all $i \geq 0$, then
$H^i(\frakg, M)$ is zero (resp.\ finite-dimensional) for all $i \geq 0$.
\end{defn}

\begin{conj} \label{conj:lie algebra cohom}
Let $K$ be a nonarchimedean field of mixed characteristics $(0,p)$.
Let $\frakg$ be a finite-dimensional Lie algebra over $K$.
Let $\frakh$ be a nilpotent Lie subalgebra of $\frakg$.
Let $M$ be a Banach space over $K$ equipped with a $K$-linear action of $\frakg$.
\begin{enumerate}
    \item[(a)] Suppose that $H^i(\frakh, M) = 0$ for all $i \geq 0$.
    Then $H^i(\frakg, M) = 0$ for all $i \geq 0$.
    \item[(b)] Suppose that $\dim_K H^i(\frakh, M) < \infty$ for all $i \geq 0$.
    Then $\dim_K H^i(\frakg, M) < \infty$ for all $i \geq 0$.
\end{enumerate}
\end{conj}

As further evidence in favor of Conjecture~\ref{conj:lie algebra cohom}, particularly part (a), we cite the following result.

\begin{thm}[Kostant] \label{thm:Kostant}
Let $\frakh$ be the nilpotent radical of a semisimple Lie algebra $\frakg$ over a field of characteristic $0$. Then for every nonzero finite-dimensional representation $V$ of $\frakg$, $\bigoplus_i H^i(\frakh, V) \neq 0$.
\end{thm}
\begin{proof}
Since $\frakg$ is semisimple, we may reduce at once to the case where $V$ is irreducible.
In this case, the claim follows from Kostant's explicit calculation of the spaces $H^i(\frakh, V)$ \cite[Corollary~5.15]{Kostant}.
\end{proof}

\bibliographystyle{plain}
\bibliography{RefList}

\end{document}